\documentclass[12pt]{amsart} 
\usepackage{amssymb} 
\usepackage{amsfonts} 
\usepackage{amsmath} 
\usepackage{amsthm} 
\usepackage{array} 
\usepackage{geometry} 
\usepackage{graphicx} 
\usepackage{mathrsfs} 
\usepackage{pifont} 
\usepackage[all,dvips,ps,cmtip]{xy} 
\usepackage{verbatim} 
\usepackage{stmaryrd}

\theoremstyle{plain} 
\newtheorem{thm}{Theorem}[subsection] 
\newtheorem{cor}[thm]{Corollary} 
\newtheorem{lem}[thm]{Lemma} 
\newtheorem{prop}[thm]{Proposition} 
\theoremstyle{definition} 
\newtheorem{defn}[thm]{Definition} 
\theoremstyle{remark} 
\newtheorem*{rem}{Remark} 
\newcommand{\Z}{\mathbb{Z}}     
\newcommand{\Q}{\mathbb{Q}} 
\newcommand{\PP}{\mathbb{P}}    
\newcommand{\sA}{\mathcal{A}}
\newcommand{\sB}{\mathcal{B}}
\newcommand{\sC}{\mathcal{C}}

\newcommand{\sF}{\mathcal{F}}
\newcommand{\sG}{\mathcal{G}}
\newcommand{\sH}{\mathcal{H}}
\newcommand{\sI}{\mathcal{I}}
\newcommand{\sK}{\mathcal{K}}

\newcommand{\sO}{\mathcal{O}}

\newcommand{\sS}{\mathcal{S}}
\newcommand{\sX}{\mathcal{X}}
\newcommand{\sY}{\mathcal{Y}}
\newcommand{\fC}{\widetilde{\mathcal{C}}}

\newcommand{\fK}{\widetilde{\mathcal{K}}}

\newcommand{\lra}{\longrightarrow}
\newcommand{\vir}{\mathrm{vir}}
\newcommand{\virclass}[1]{{[{#1}]^\vir}}            
\newcommand{\KSTACK}[4]{{\sK_{#1,#2}(#3,#4)}}   
\newcommand{\kstack}{\KSTACK{g}{n}{X}{\beta}}   
\newcommand{\VKstack}[4]{{\fK_{#1,#2}(#3,#4)}}  
\newcommand{\vkstack}{\VKstack{g}{n}{X}{\beta}} 
\newcommand{\CIS}[2]{{\sI_{#1}(#2)}}            
\newcommand{\cis}{\CIS{\mu}{X}}         
\newcommand{\RCIS}[2]{{\overline{\sI}_{#1}(#2)}}    
\newcommand{\rcis}{\RCIS{\mu}{X}}           
\newcommand{\thickslash}{\mathbin{\!\!\pmb{\fatslash}}} 

\begin{document}
\title{Very Twisted Stable Maps}
\date{\today} 
\author[Q. Chen]{Qile Chen}
\address{Qile Chen, Department of Mathematics, Box 1917, Brown University,
Providence, RI, 02912, U.S.A}
\email{q.chen@math.brown.edu}
\author[S. Marcus]{Steffen Marcus}
\address{Steffen Marcus, Department of Mathematics, Box 1917, Brown University,
Providence, RI, 02912, U.S.A}
\email{ssmarcus@math.brown.edu}
\author[H. \'Ulfarsson]{Henning \'Ulfarsson}
\address{Henning \'Ulfarsson, School of Computer Science, Reykjavik University,
\mbox{Reykjavik}, Iceland}
\email{henningu@ru.is}
\begin{abstract}
Let $X$ be a smooth projective Deligne-Mumford stack over an algebraically closed field $k$ of characteristic $0$.  In this paper we construct the moduli stack $\widetilde{\sK}_{g,n}(X,\beta)$ of \emph{very twisted stable maps}, extending the notion of twisted stable maps from \cite{AV} to allow for generic stabilizers on the source curves.  We also consider the Gromov-Witten theory given by this construction.
\end{abstract}
\maketitle
\setcounter{tocdepth}{1}
\tableofcontents
%
%
\section{Introduction}
\label{sec:1}


Throughout, we let $X$ be a smooth, projective Deligne-Mumford stack (DM-stack from now on)
over an algebraically closed field $k$ of characteristic 0.

The Gromov-Witten theory of orbifolds was first introduced in the symplectic setting in \cite{CR}.  This was followed by an adaptation to the algebraic setting in \cite{AGV2} and \cite{AGV}, where the Gromov-Witten theory of DM-stacks was developed, and heavy use is made of the moduli stack of twisted stable maps into $X$, denoted $\kstack$.  This stack was constructed in \cite{AV} and is the necessary analogue of Kontsevich's moduli stack of stable maps for smooth projective varieties when replacing the variety with a DM-stack.  The main purpose of this note is to provide a further extension of these spaces by allowing generic stabilizers on the source curves of the twisted stable maps.

Following \cite{AGV}, we have a diagram:
\[
\xymatrix{
     **[l]\Sigma_i^\sC \subset \sC \ar[r]^f \ar[d] & X\\
     \kstack\\}
\]
where $\Sigma_i^\sC \subset \sC \overset{f}{\longrightarrow} X$ is the universal $n$-pointed twisted stable map.  This gives rise to evaluation maps $e_i:\kstack\lra\rcis$ mapping into the rigidified cyclotomic inertia stack $\rcis$ of $X$. If $\gamma_1,\ldots,\gamma_n \in A^\ast(\rcis)_\Q$ then the Gromov-Witten numbers are defined to be
\[
\left<\gamma_1,\ldots,\gamma_n\right>^X_{g,\beta} = \int\limits_\virclass{\kstack}\prod_i e_i^\ast\gamma_i
\]
where $\virclass{\kstack}$ is the virtual fundamental class of $\kstack$ as in \cite{BF}.

In the case when $X$ is a $3$-dimensional Calabi-Yau {\em variety}, we also have Donaldson-Thomas theory (originating from \cite{T},\cite{DT}) which in contrast to Gromov-Witten theory gives invariants by counting sheaves.  We get the following diagram:
\[
\xymatrix{
**[r] \sY \subset  X\times\sH\text{ilb}_{\chi,\beta}(X) \ar[d] &\\
\sH\text{ilb}_{\chi,\beta}(X) & \\}
\]
where the virtual dimension of $\sH\text{ilb}_{\chi,\beta}(X)$ is zero as in \cite{BF}.  The conjectural Donaldson-Thomas / Gromov-Witten correspondence of \cite{MNOP} predicts a correspondence (in the case n=0):
\[
\xymatrixcolsep{5pc}\xymatrix{\Bigg\{\displaystyle\int\limits_{\virclass{\sK_{g,0}(X,\beta)}}\!\!\!\!\!\!\!\!\!\!\!1\,\,\,\,\,
=GW(g,\beta)\Bigg\} \ar@{<->}[r] &\Bigg\{\displaystyle \int\limits_{\virclass{\sH\text{ilb}_{\chi,\beta}(X)}} \!\!\!\!\!\!\!\!\!\!\!\!1\,\,\,\,\,\, = DT(\chi,\beta)\Bigg\}}.
\]
This correspondence is manifested in a subtle relationship between generating functions of the invariants.
One wishes to discover a similar correspondence when $X$ is a 3-dimensional Calabi-Yau {\em orbifold}. The Hilbert scheme of a stack was constructed by Olsson and Starr \cite{OS03}, but notice that in general it contains components corresponding to substacks with nontrivial generic stabilizers.
In our definition of $\kstack$ the twisted curves used as the sources of our maps have stacky structure only at the nodes and marked points.  To allow for the above correspondence, one needs to extend the notion of the space of twisted stable maps as defined in \cite{AV} and \cite{AGV} so that our twisted curves have generic stabilizers.  We approach this problem in three steps:

\begin{enumerate}
    \renewcommand{\labelenumi}{(\arabic{enumi})}
    \item Construct the stack $\sG_X$ of \'etale gerbes in $X$ as a rigidification of the stack $\sS_X$ of subgroups of the inertia stack $\sI(X)$.  We exhibit $\sS_X$ as the universal gerbe sitting over $\sG_X$, giving a diagram:
     \[
    \xymatrix{
    \sS_X \ar[d]^\alpha \ar[r]^\phi & X\\
    \sG_X.}
    \]
    This is done in section \ref{sec:2}.
    \item Define the moduli stack of very twisted stable maps $\vkstack$ by setting
    \[
    \vkstack := \coprod_{\beta_\sG} \KSTACK{g}{n}{\sG_X}{\beta_\sG}
    \]
    where the disjoint union is taken over all curve classes $\beta_\sG\in H^\ast(\sG_X)$ such that $\phi_\ast\alpha^\ast\beta_\sG = \beta$.  By pulling back, we see that each $\KSTACK{g}{n}{\sG_X}{\beta_\sG}$ has two different universal objects sitting above it, one giving twisted stable maps into $\sG_X$ and the other giving ``very twisted" stable maps into $X$.  Through our disjoint union, we get the two corresponding universal objects sitting above $\vkstack$.  This is done in section \ref{sec:3}.
    \item Relate the Gromov-Witten theory of $X$ given by the two different universal objects sitting above $\vkstack = \coprod_{\beta_\sG} \KSTACK{g}{n}{\sG_X}{\beta_\sG}$ and show that they give the same invariants.  This is done in section \ref{sec:4}.
\end{enumerate}

\subsection{Acknowledgements}
We would like to thank Jim Bryan and Rahul Pandharipande for suggesting the problem, and Angelo Vistoli for his great help in understanding the constructions of section \ref{sec:2}.  We would also like to thank our advisor, Dan Abramovich, for his continuous support and encouragement.

Finally, Ezra Getzler has suggested earlier that very twisted curves should be of independent interest; this note may give further motivation for following this suggestion.
%
%

\section{Constructing the stacks $\sS_X$ and $\sG_X$}
\label{sec:2}


In this section we construct DM-stacks $\sS_X$ and $\sG_X$ and show that $\sG_X$ is presented as the rigidification of $\sS_X$ along a group scheme.  Our construction follows closely the construction of $\cis$ and its rigidification,
from \cite{AGV}.

%
%
\subsection{The stack of subgroups of the inertia stack}

Recall that $X$ is a smooth, projective DM-stack over an algebraically closed field $k$ of characteristic $0$.  We begin by defining the stack $\sS_X$ of subgroups of the inertia stack $\sI(X)$.

\begin{defn}\label{defn:Sx}
We define a category $\sS_{X}$, fibered over the category of schemes, as follows:
\begin{enumerate}
      \renewcommand{\labelenumi}{(\alph{enumi})}
         \item An object of $\sS_{X}(T)$ consists of a pair $(\xi,\alpha: G \hookrightarrow \sA ut_{T}(\xi))$, where $\xi \in X(T)$, the arrow $\alpha$ is an injective morphism of group schemes, and $G$ is finite and \'etale over $T$.
         \item An arrow sitting over $T \lra T'$ from $(\xi, \alpha)\in\sS_X(T)$  to $(\xi',\alpha')\in\sS_X(T')$ is a morphism $F: \xi \lra \xi'$ making the following diagram commutative:
           \[
           \xymatrix{
              G\ar[d]^\alpha \ar[r] & G'\ar[d]^{\alpha^{\prime}}\\
               \sA ut_{T}(\xi) \ar[d] \ar[r] & \sA ut_{T^{\prime}}(\xi^{\prime}) \ar[d]\\
               T \ar[r]                    & T^{\prime}}
            \]
            where $\sA ut_{T}(\xi)\lra \sA ut_{T'}(\xi')$ is the morphism induced by $F$ and the diagram is cartesian.
\end{enumerate}
\end{defn}

This definition follows Definition 3.1.1 in \cite{AGV} except we are allowing our subgroup to vary.

There are several things that need mention or proof.

\begin{rem}
Notice that there is an obvious morphism (of fibered categories) $\sS_{X}\lra X$ which sends $(\xi, \alpha)$ to $\xi$.
\end{rem}

\begin{prop}
$\sS_{X}$ is fibered in groupoids over the category of schemes.
\end{prop}
\begin{proof}
By Proposition 3.22 in Part 1 of \cite{FGIKNV} $\sS_X$ is fibered in groupoids if and only if
\begin{itemize}
\item every arrow is cartesian; and
\item given $\xi'$ over $T'$ and $f : T \to T'$ there exists an arrow $\phi : \xi \to \xi'$ over $f$.
\end{itemize}
The first condition is automatically satisfied since every arrow in $\sS_X$ is an arrow in $X$ which
is fibered in groupoids. For the second we set $\xi = f^*\xi'$ and $G = f^*G'$, so $\alpha$ is automatically
defined. We now clearly have an arrow $(\xi,\alpha) \to (\xi',\alpha')$ making the necessary diagram
commute.\end{proof}

\begin{prop}\label{prop:SXisDM}
The category $\sS_{X}$ is a DM-stack, and the functor $\sS_{X}\lra X$ is representable and finite.
\end{prop}
\begin{proof}
We prove that the functor $\sS_{X}\lra X$ is representable and finite, which implies that
$\sS_{X}$ is a DM-stack.
The inertia stack $\sI(X) \to X$ is finite, unramified and representable, since $X$ is assumed to be
a separated DM-stack. We can therefore let $M$ be the maximal degree of a fiber of $\sI(X)$. We
begin by showing that the relative Hilbert scheme $\sH ilb(\sI(X)/X)$ is unramified and finite:
\begin{itemize}
\item The fiber over a geometric point of $X$ is unramified since a subscheme of a
finite \'etale scheme has only trivial infinitesimal deformations.
\item The relevant Hilbert polynomials here are just integers since the schemes in question are all finite.
These integers are bounded above by $M$, so there are only finitely many of them. This together with the
fact that each component is proper implies finiteness.
\end{itemize}
The relative Hilbert scheme is representable by definition. The last part we need is to notice that
$\sS_X$ is an open and a closed subscheme in $\sH ilb(\sI(X)/X)$ since the condition on a finite subset to be a group is an open and closed condition.
\end{proof}

\begin{prop}
 \begin{enumerate}
   \item When $X$ is smooth, $\sS_{X}$ is smooth as well.
   \item When $X$ is proper, $\sS_{X}$ is proper as well.
 \end{enumerate}
\end{prop}
\begin{proof}
We start with (1).  Let $x$ be a geometric point of $X$ and $G_x$ its stabilizer group.  We know from \cite{AOV} that we can view $X$ in a local chart around $x$ as $[U/G_x]$.  $\sS_X$ has a local chart $[U^H/N(H)]$, where $H\subset G_x$ is a subgroup and we quotient out by the normalizer subgroup $N(H)$. To prove that this chart is smooth note that $T_{U^H}=(T_U)^H$ and that $\dim (T_U)^H = \dim (T_U^*)^H$ and
$T_V^*$ is generated by
\[
\overline{y_1}, \dotsc, \overline{y_d}, \overline{y_{d+1}}, \dotsc, \overline{y_n},
\]
where $\overline{y_1}, \dotsc, \overline{y_d}$ are the invariant generators. Since $H$ is reductive
we have a section of the projection map $\mathfrak{m} \to \mathfrak{m}/\mathfrak{m}^2$ which allows us
to lift all the generators to
\[
y_1, \dotsc, y_d, y_{d+1}, \dotsc, y_n,
\]
in such a manner that $y_1, \dotsc y_d$ are still the invariant generators and $y_{d+1}, \dotsc, y_n$ span a finite-dimensional representation. Let $J$ be the ideal generated
by $y_{d+1}, \dotsc, y_n$. Then the quotient by that ideal corresponds to an invariant subscheme, which is smooth since
formally it is generated by $y_{1}, \dotsc, y_d$. This proves that the dimension of $U^H$ is at least, and thus equal, to
$\dim T_{U^H}$.\

(2) follows from the proof of Proposition \ref{prop:SXisDM}, since $\sS_X$ is a closed subscheme
of $\sH ilb_X(\sI(X))$, which is proper over $X$.
\end{proof}

%
%
%
\subsection{Alternative description of $\sS_{X}$}

It will be useful to provide another, less obvious, description of $\sS_X$.  We start with a definition.

\begin{defn}\label{gerbe}
An \emph{\'etale gerbe} over a scheme $S$ is an DM-stack $\sA$
over $S$ such that
\begin{enumerate}
      \item there exists an \'etale covering {$S_{i} \lra S$} such that each $\sA(S_{i})$ is not empty.
      \item given two objects $a$ and $b$ of $\sA(T)$, where $T$ is an $S$-scheme, there exists a covering {$T_{i} \lra T$} such that the pullbacks $a_{T_{i}}$ and $b_{T_{i}}$ are isomorphic in $F(T_{i})$.
\end{enumerate}
More generally, a stack $\sF$ over a stack $\sX$ is an \'etale
gerbe if for any morphism $V\lra\sX$ with $V$ a scheme, the
pull-back $\sF_{V}\lra V$ along $V\lra\sX$ is an \'etale gerbe.
\end{defn}


\begin{defn}\label{S'x}
We define a 2-category $\sS_{X}'$, fibered over the category of
schemes, as follows:
\begin{enumerate}
      \renewcommand{\labelenumi}{(\alph{enumi})}
      \item An object of $\sS_{X}'(T)$ consist of representable morphisms $\phi : \sA \lra X$, where $\sA$ is
               an \'etale gerbe over T, with a section $\tau : T\lra \sA$.
               \[\xymatrix
              {\sA \ar[r]^{\phi}\ar[d]  & X\\
               T. \ar^{\tau}@/^/[u]}\]

    \item A 1-arrow $(F, \rho): (\sA, \phi)\lra(\sA^{\prime},\phi^{\prime})$ consists of a morphism $F: \sA \lra \sA^{\prime}$
    over some $f:T \lra T^{\prime}$ making a cartesian square,
    and a natural transformation $\rho:\phi\Rightarrow\phi^{\prime}\circ F$
    making the following diagram commutative:
    \[
    \xymatrix{
                &&X \\
        \sA \ar@/^/[urr]^{\phi} \ar[r]_F \ar[d]& \sA' \ar@/_/[ur]_{\phi'} \ar[d]&\\
        T  \ar[r] &T'. \\}
    \]

    \item A 2-arrow $(F,\rho) \lra (F_{1},\rho_{1})$ is an natural transformation $\sigma:F \Rightarrow F_{1}$ giving an equivalence, and compatible with $\rho$ and $\rho_{1}$ in the sense that the following diagram is commutative:
    \[
    \xymatrix{&\phi\ar[ld]_{\rho} \ar[rd]^{\rho_{1}}&\\
          \phi\circ F \ar[rr]^{\phi(\sigma)} && \phi^{\prime}\circ F_{1}.}
          \]
\end{enumerate}
\end{defn}

\begin{rem}

\begin{enumerate}
\item By Lemma 3.21 in \cite{LMB}, the section $\tau$ gives a
neutral section, hence the gerbe $\sA$ is isomorphic to
$\sB_{T}G_{\tau}$ where $G_{\tau}=\sA ut_{T,\sA}(\tau,\tau)$ is
\'etale over T.  The section $\tau$ corresponds to the trivial
$G_{\tau}$-torsor $G_{\tau}\lra T$.

\item Since T gives a moduli space of $\sA$, the arrow
$\sA\lra T$ is proper. This implies that the group scheme
$G_{\tau}$ is finite over T.
\end{enumerate}
\end{rem}

%

\begin{lem}
The 2-category $\sS'_{X}$ is equivalent to a category.
\end{lem}
\begin{proof}
    Lemma 3.3.3 in \cite{AGV} shows that the isomorphism group of a 1-arrow in $\sS'_X$ is trivial.  Since all 2-arrows are isomorphisms, the result follows.
\end{proof}

\begin{defn}
By abuse of notation, we denote by $\sS'_{X}$ the 1-category associated to the above 2-category.  Our arrows become 2-isomorphism classes of 1-arrows.
\end{defn}

\begin{defn}
We define a morphism of fibered categories
$\sS_{X}'\lra \sS_{X}$ as follows:
\begin{enumerate}
      \renewcommand{\labelenumi}{(\alph{enumi})}
      \item Given an object
            \[
            \xymatrix
              {\sA \ar[r]^{\phi}\ar[d]  & X\\
               T}
              \]
            with a section $\tau: T\lra\sA$, we obtain a pair $(\xi, \alpha)$ as follows:  $\xi$ is obtained by composing $\phi$ with $\tau$. Note that $\tau$ gives an element $\zeta_{\tau, T}\in\sA(T)$.  Then $\alpha$ is the associated map of automorphisms
            \[
            G=\sA ut_{T}(\zeta_{\tau, T})\lra \sA ut_{T}(\xi),
            \]
            which is injective since $\phi$ is representable.
      \item Given an arrow $\rho$ as above, we obtain an arrow $F:\phi(\zeta_{\tau,T})\lra\phi^{\prime}(\zeta_{\tau^{\prime},T^{\prime}})$ by
            completing the following diagram:
       \[
       \xymatrix{ \phi(\zeta_{\tau, T})\ar[d]_{\rho} \ar@{.>}[rrdd]^{F}            &&  \\
                    \phi^{\prime}\circ f_{*}(\zeta_{\tau, T})\ar@{=}[d]                  &&  \\
                    \phi^{\prime}((\zeta_{\tau^{\prime}, T^{\prime}})_{T})\ar[rr] && \phi^{\prime}(\zeta_{\tau^{\prime}, T^{\prime}}).
         }
         \]
\end{enumerate}
\end{defn}

The proof of the following proposition is almost the same as
Proposition 3.2.3 in \cite{AGV}.  For completeness we rewrite
it in our case.

\begin{prop}
The morphism $\sS_{X}^{\prime}\lra \sS_{X}$ is an equivalence of fibered categories.
\end{prop}
\begin{proof}
By Proposition 3.36 in Part 1 of \cite{FGIKNV}, it is enough to show that
the induced functor on the fiber $\sS'_{X}(T)\lra\sS_{X}(T)$ is an
equivalence for any given scheme T.
\begin{enumerate}

\item \textbf{The functor is faithful.}

      Assume we are given two elements in $\sS'_{X}(T)$ and a 2-arrow $\rho:\phi\Rightarrow\phi^{\prime}\circ F$ making the following diagram commutative:
    \[
    \xymatrix{
                &&X \\
        \sA \ar@/^/[urr]^{\phi} \ar[r]_F \ar[d]& \sA' \ar@/_/[ur]_{\phi'} \ar[d]&\\
        T \ar[r] &T. \\}
    \]
      We need to show that for any T-scheme $U$ and any
      $G_{\tau}\times_{T}U$-torsor $P\lra U$, the arrow
      $\rho(P\lra U): \phi(P\lra U)\lra\phi'\circ F(P\lra U)$
      is uniquely determined by $\rho(G_{\tau}\lra T): \phi(G_{\tau}\lra T)\lra\phi'\circ F(G_{\tau}\lra
      T)$. Note that $G_{\tau}\lra T$ is the trivial torsor given by the section
      $\tau$.

      Let $\{U_{i}\lra U\}$ be an \'etale covering such
      that the pull-backs $P_{i}\lra U_{i}$ are trivial (i.e.,
      $P_{i}\backsimeq (G_{\tau}\times_{T}U)\times_{U}U_{i}\backsimeq
      G_{\tau}\times_{T}U_{i}$). Since $X$ is a DM-stack, the arrow
      $\phi(P\lra U)\lra\phi'\circ F(P\lra U)$ is determined by $\phi(P_{i}\lra U_{i})\lra\phi'\circ F(P_{i}\lra
      U_{i})$. Hence, we can assume that $P\lra U$ is trivial. But
      then the cartesian arrow $P\lra G_{\tau}$ and the induced
      diagram
      $$\xymatrix{
            \phi(P\lra U) \ar[r]\ar[d] & \phi'\circ F(P\lra U)\ar[d]\\
            \phi(G_{\tau}\lra T) \ar[r]  & \phi'\circ
            F(G_{\tau}\lra T)
      }$$ proves what we want.

\item \textbf{The functor is fully faithful.}

      Assume we are given an arrow $\beta: \phi(G_{\tau}\lra T)\lra \phi'\circ F(G_{\tau}\lra
      T)$ in $X(T)$ which is compatible with the action of
      $G_{\tau}$ and $G_{\tau'}$. First, consider the trivial torsor
      $P\simeq G_{\tau}\times_{T}U\lra U$, which induces a
      cartesian arrow $P\lra G_{\tau}$.  By the definition of a cartesian arrow, there is a unique arrow $\rho(P\lra U)$ that we can insert in the diagram
      $$\xymatrix{
            \phi(P\lra U) \ar@{-->}[rr]^-{\rho(P\lra U)}\ar[d] && \phi'\circ F(P\lra U)\ar[d]\\
            \phi(G_{\tau}\lra T) \ar[rr]^{\beta}     && \phi'\circ
            F(G_{\tau}\lra T)
      }$$ making it commutative. This arrow $\rho(P\lra U)$ is
      independent of the chosen trivialization, since $\beta$
      compatible with the group action.

      Now, if the $G_{\tau}$-torsor $P\lra U$ is not necessarily
      trivial, choose a covering ${U_{i}\lra U}$ such that the
      pull-backs $P_{i}\lra U_{i}$ are trivial. We have arrows $\phi(P_{i}\lra U_{i})\lra\phi'\circ F(P_{i}\lra
      U_{i}))$ in $X(U_{i})$, and their pullbacks to
      $U_{i}\times_{U}U_{j}$ coincide; hence they glue together to
      given an arrow $\rho(P\lra U): \phi(P\lra U)\lra\phi'\circ F(P\lra
      U)$. It is not hard to see that $\rho(P\lra U)$ does not
      depend on the choice of covering, and defines a 2-arrow $\phi\lra\phi'\circ
      F$ whose image in $\sS_X(T)$ coincides with $\beta$.

\item \textbf{The functor is essentially surjective.}

      Note that since $G_{\tau}$ is \'etale over T, and $\sA\simeq
      \sB_{T}G_{\tau}$, then $T\lra \sA$ gives an atlas of $\sA$. Given an object $(\xi,\alpha: G \hookrightarrow \sA
      ut_{T,X}(\xi))$ of $\sS_{X}(T)$, we need to construct $\phi: \sB_{T}G\lra
      X$, whose image in  $\sS_{X}(T)$ is isomorphic to
      $(\xi,\alpha)$.

      Let $P\lra U$ be a $G\times_{T}U$-torsor, where $U$ is a
      T-scheme. The morphism $G\times_{T}P\lra P\times_{U}P$ given
      by $(g,p)\mapsto(gp,p)$, is an isomorphism. The
      pull-back of $\xi$ to $P\times_{U}P\simeq G\times_{T}P$ via
      the first and the second projection coincide with the
      pull-back $\xi_{G\times_{T}P}$, and the isomorphism is given
      by the projection $G\times_{T}P\lra G$. These give descent
      data for $\xi_{P}$ along the \'etale covering $P\lra
      U$. The descent data are effective and define an object
      $\eta$ of $X(U)$, so we assigned to every object of
      $\sB_{T}G(U)$ an object of $X(U)$.  This extends to a morphism
      of fibered categories $\phi: \sB_{T}G\lra X$.

      Now let $G\lra T$ be the trivial G-torsor over T. We claim
      that $\eta:=\phi(G\lra T)$ of $X(T)$ is isomorphic to $\xi$.
      In fact, the object with descent data defining $\xi$ is
      $\xi_{G}$, with the descent data given by the identity on
      $\xi_{G\times G}$. Those descend to $\eta\simeq\xi$ in
      $X(T)$. Hence the image of $\phi$ in $\sS_{X}(T)$ is
      isomorphic to $(\xi, G)$, as we wanted.
\end{enumerate}
\end{proof}

%
%
\subsection{The stack of gerbes in $X$}

In the following we introduce a stack $\sG_{X}$ which is closely related to $\sS_{X}$ and will play an important role in the construction of the moduli stack of very twisted stable maps.

\begin{defn}
Define as before a 2-category $\sG_{X}$ with a functor to the category of
schemes as follows:
\begin{enumerate}
      \renewcommand{\labelenumi}{(\alph{enumi})}
    \item An object over a scheme T is a pair $(\sA, \phi)$ where $\sA$ is an \'etale gerbe over T
    and $\sA\lra X$ is a representable morphism.

    \item A 1-arrow $(F, \rho): (\sA, \phi)\lra(\sA^{\prime},\phi^{\prime})$ consists of a morphism $F: \sA \lra \sA^{\prime}$
    over some $f:T \lra T^{\prime}$ making a cartesian square,
    and a natural transformation $\rho:\phi\Rightarrow\phi^{\prime}\circ F$
    making the following diagram commutative:
    \[
    \xymatrix{
                &&X \\
        \sA \ar@/^/[urr]^{\phi (rep.)} \ar[r]_F \ar[d]& \sA' \ar@/_/[ur]_{\phi' (rep)} \ar[d]&\\
        T \ar[r] &T'.\\}
    \]
    \item A 2-arrow $(F,\rho) \lra (F_{1},\rho_{1})$ is a natural transformation $\sigma:F \Rightarrow F_{1}$ giving an equivalence, and compatible with $\rho$ and $\rho_{1}$ in the sense that the following diagram is commutative:
  $$\xymatrix{&\phi\ar[ld]_{\rho} \ar[rd]^{\rho_{1}}&\\
              \phi\circ F \ar[rr]^{\phi(\sigma)} && \phi^{\prime}\circ F_{1}.}$$
\end{enumerate}
\end{defn}

\begin{defn}
Exactly as above, the 2-category $\sG_{X}$ is equivalent to a 1-category.  By abuse of notation,
we denote by $\sG_{X}$ its associated 1-category.
\end{defn}

It will follow from our results below on rigidification that $\sG_{X}$ is a DM-stack.

\begin{rem}
There is a tautological functor of categories $\sS_{X}\lra \sG_{X}$, sending the pair $(\phi, \tau)$ in $\sS_X(T)$
to the representable morphism $\phi: \sA\lra X$ in $\sG_{X}(T)$.
\end{rem}

%
%
\subsection{Rigidification}

Rigidification of algebraic stacks was first defined and used in \cite{ACV}, \cite{AOV} and \cite{Ro}.  It was foreseen in \cite{A}.

In this section, we discuss the rigidification of DM-stacks, and prove a proposition that the morphism  $\sS_{X}\lra\sG_{X}$ defined in the above remark gives a rigidification of $\sS_X$ along an appropriate relative group scheme, which all our constructions rely on.

We begin by restating the theorem in \cite{AOV} on the existence of the rigidification of an algebraic stack.  We stick to the case of DM-stacks since this is all we will need for our purposes.  It follows from Theorem 5.1.5 in \cite{ACV} that the rigidification of a DM-stack is a DM-stack, and the quotient map is \'etale.

Let $S$ be a scheme, $X\lra S$ a DM-stack of finite type over
$(Sch/S)$.  Suppose $G\subseteq\sI(X)$, where $\sI({X})$ is the
inertia stack of $X$, and G is a subgroup stack of $\sI({X})$ of finite type and \'etale over $X$.

\begin{thm}[\cite{AOV} Theorem A.1]\label{rigid}
There exists a finite type DM-stack $X\thickslash G$ over S with a
morphism $\rho:X\lra X\thickslash G$ satisfying the following
properties:
\begin{enumerate}
    \item $X$ is an \'etale gerbe over $X\thickslash G$.
    \item For each object $\xi$ of $X(T)$, the morphism of group schemes
    \[
    \rho: \sA ut_{T}(\xi)\lra \sA ut_{T}(\rho(\xi))
    \]
  is surjective with kernel $G_{\xi}$.
\end{enumerate}
Furthermore, if $G$ is finite over $X$ then $\rho$ is proper;
while since G is \'etale, $\rho$ is also \'etale. These
properties characterize $X\thickslash G$ uniquely up to
equivalence.
\end{thm}

\begin{rem}
From (\cite{AOV} Theorem A.1) we already know that $X\thickslash
G$ is an algebraic stack. To show that it is also a DM-stack,
we only need an \'etale atlas. But since the arrow $\rho:X\lra
X\thickslash G$ is $\acute{e}tale$ and surjective, the atlas of
$X\thickslash G$ is obtained from that of $X$.
\end{rem}

We add to this a proposition about pulling back rigidifications along representable morphisms.

\begin{prop}
Let $X$ and $G$ be defined as above, and $Y$ a DM-stack of finite type over $(Sch/S)$. Given any representable morphism $Y\lra X\thickslash G$, the pullback
    \[
    \tilde{\rho}: Y\times_{X\thickslash G}X \lra Y
    \]
    of $\rho$ satisfies the properties in the above theorem, hence
    is a rigidification of $Y\times_{X\thickslash G}X$ along the
    pull-back $G_{Y}:=(Y\times_{X\thickslash G}X)\times_{X}G$.
\end{prop}

\begin{proof}
First, any $\sA\lra T$ which is a pullback of
$Y\times_{X\thickslash G}X \lra Y$ along $T\lra Y$ can be viewed
as a pullback of $X\lra X\thickslash G$ along the composition
$T\lra Y\lra X\thickslash G$. Hence, by Definition \ref{gerbe},
$\tilde{\rho}: Y\times_{X\thickslash G}X \lra Y$ is an \'etale
gerbe.

Second, given any $T\lra Y\times_{X\thickslash G}X$, where T is
any scheme, consider the following cartesian diagram:
$$\xymatrix{
  G_{T}\ar[rr]\ar[d] && G_{Y}\ar[rr]\ar[d]         && G\ar[d]\\
  T\ar[rr]^-{(\zeta,\xi,\mu)}  && Y\times_{X\thickslash G}X \ar[rr]^{\psi}\ar[d]_{\tilde{\rho}} && X \ar[d]^{\rho}\\
                    && Y \ar[rr]^{\phi}     && X\thickslash G
  }$$ where $\xi=\psi\circ(\zeta,\xi,\mu)\in X(T)$, $\zeta=\tilde{\rho}\circ(\zeta,\xi,\mu)\in
  Y(T)$ and $\mu: \phi\circ\zeta\simeq\rho\circ\xi$ is an isomorphism in $X\thickslash
  G(T)$. Note that $G_{T}\hookrightarrow\sA ut_{T}(\xi)$ gives the action of $G_{T}$ on $\xi$,
  and lies in the kernel of $\rho: \sA ut_{T}(\xi)\lra\sA
  ut_{T}(\rho(\xi))$. There is an natural action of $G_{T}$
  on $(\zeta,\xi,\mu)$, which acts on $\xi$ and fixes $\zeta$ and
  $\mu$. Hence there is a injection $G_{T}\hookrightarrow\sA
  ut_{T}((\zeta,\xi,\mu))$. Since $G_{T}$ is \'etale and of finite type over T, this implies that
$G_{Y}$ is a subgroup stack of the inertia stack
$I_{Y\times_{X\thickslash G}X}$, and \'etale and finite type over
$Y\times_{X\thickslash G}X$.

Finally, consider the following diagram:
$$\xymatrix{
\emph{id}\ar[r]& G_{T}\ar@{^{(}->}[r] & \sA ut_{T}(\xi)\ar@{>>}[r]
&
\sA ut_{T}(\rho(\xi))\ar[r] & \emph{id} \\
&G_{T} \ar@{^{(}->}[r]\ar[u]^{\simeq} & \sA ut_{T}((\zeta,\xi,\mu))
\ar@{>>}[r]\ar[u]^{f} & \sA ut_{T}(\tilde{\rho}(\zeta,\xi,\mu)).
\ar[u]^{g} }$$ Here, $f$ and $g$ are injective since the
corresponding arrow is representable and the first row is exact.
An easy diagram chase shows that the second row is also exact,
which implies that $G_{T}$ is the kernel of $\sA
ut_{T}((\zeta,\xi,\mu))\lra \sA
ut_{T}(\tilde{\rho}(\zeta,\xi,\mu))$.

By the uniqueness of the theorem above, $\tilde{\rho}:
Y\times_{X\thickslash G}X \lra Y$ is the rigidification along
$G_{Y}\lra Y$.
\end{proof}

Now let's consider the stack $\sS_{X}$. For each object $(\xi,G)$
of $\sS_{X}(T)$ over a scheme T, we associate a
subgroup scheme $G\subset \sA ut_{T}((\xi,G))$. Note that $G$ is
finite and \'etale over $T$. Take $(\xi,G)\lra(\xi',G)$ in
$\sS_{X}$ over $T\lra T'$. By the definition of $\sS_{X}$, $G$ is
the pullback of $G'$ along $T\lra T'$. By Appendix A of
\cite{AOV}, there is a unique \'etale finite subgroup stack,
denoted $\Gamma_{X}\subset I_{\sS_{X}}$, where $I_{\sS_{X}}$ is the
inertia stack of $\sS_{X}$, such that for any object $(\xi,G)$ of
$\sS_{X}(T)$ the pullback of $\Gamma_{X}$ to T coincides with G.

The following is the main result of this section.

\begin{prop}
There is an equivalence $\sS_{X}\thickslash\Gamma_{X}\lra \sG_{X}$
of fibered categories, hence we can view $\sG_{X}$ as a
rigidification of $\sS_{X}$ along $\Gamma_{X}\lra\sS_{X}$.
\end{prop}

\begin{proof}
Given any scheme T and an arrow
$T\lra\sS_{X}\thickslash\Gamma_{X}$, consider the following
cartesian diagram:
$$\xymatrix{\sA\ar[r]\ar[d]&\sS_{X}\ar[r]\ar[d]           &X     \\
           T\ar[r]        &\sS_{X}\thickslash\Gamma_{X},
           }$$
where $\sA$ is the gerbe obtained by pulling back along
$T\lra\sS_{X}\thickslash\Gamma_{X}$. Also, note that the arrows
$\sA\lra\sS_{X}$ and $\sS_{X}\lra X$ are representable, so their
composition $\sA\lra\sS_{X}\lra X$ is also representable. Since $\Gamma_{X}$ is finite
\'etale over $\sS_{X}$, by the above theorem the arrow $\sS_{X}\lra
\sS_{X}\thickslash\Gamma_{X}$ is \'etale and proper, hence $\sA$ is an \'etale proper gerbe over T with
a representable arrow $\sA\lra X$. This construction extends to an
arrow $\sS_{X}\thickslash\Gamma_{X}\lra \sG_{X}$ of fibered
categories.

To show that the arrow $\sS_{X}\thickslash\Gamma_{X}\lra \sG_{X}$
constructed above is an equivalence, we construct its inverse.
Given any object in $\sG_{X}(T)$
$$\xymatrix{\sA\ar[r]^{\phi}\ar[d]^{\pi} & X\\
            T,                            &
            }$$
consider a $T$-scheme $U$, and an arrow $\zeta: U\lra \sA$. Denoting
$\phi(\zeta)$ by $\xi$ for simplicity, since $\phi$ is representable we have an
injection $\alpha: \sA ut_{U,\sA}(\zeta)\hookrightarrow\sA
ut_{U,X}(\xi)$. Hence we have a pair $(\xi,\alpha)$ which is an 
object in $\sS_{X}(U)$. Note that this construction gives an arrow
of stacks $\psi: \sA\lra\sS_{X}$. From this construction, it
is easy to see that the composition $\sA\lra\sS_{X}\lra X$ is the
representable arrow $\phi$.  By Lemma 1.3 in \cite{V}, $\psi$ is
also representable.

Next we will construct an arrow $T\lra
\sS_{X}\thickslash\Gamma_{X}$ and show that $\sA\lra T$ is the
pullback of $\sS_{X}\lra \sS_{X}\thickslash\Gamma_{X}$ along this
arrow, proving our result.

According to the definition of a gerbe, we can choose a covering
$\{T_{i}\}$ of T such that there exists $\zeta_{i}\in
\sA(T_{i})$ non-zero and isomorphisms $f_{ij}: \zeta_{i}|_{T_{ij}}\simeq
\zeta_{j}|_{T_{ij}}$, where $T_{ij}\simeq T_{i}\times_{T}T_{j}$
and $f_{ji}=f_{ij}^{-1}$. Note that $f_{ki}\circ f_{jk}\circ
f_{ij}=g_{i}\in \sA ut_{T_{i},\sA}(\zeta_{i})|_{T_{ij}}$. Since
$g_{i}$ might not be the identity, $\{f_{ij}\}$ might not satisfy
the cocycle condition, becoming a obstruction to glue the
$\zeta_{i}$. Let us denote $\rho\circ\psi\circ\zeta_{i}$ by $\xi_{i}$, and
$\rho\circ\psi\circ f_{ij}$ by $f_{ij}'$. Note that
$f'_{ij}:\xi_{i}|_{T_{ij}}\backsimeq\xi_{j}|_{T_{ij}}$ gives an isomorphism. Now we
want to glue $\{\xi_{i}\}$, and for this we only need to show that
$f'_{ki}\circ f'_{jk}\circ f'_{ij}=\rho\circ\psi\circ
g_{i}=\emph{id}$. But we have the following exact sequence:
\[\xymatrix{1\ar[r]&\sA ut_{T_{i},\sA}(\zeta_{i})|_{T_{ij}}\ar[r]^-{\psi}&\sA ut_{T_{i},\sS_{X}}(\psi\circ\zeta_{i})|_{T{ij}}\ar[r]^-{\rho}&\sA ut_{T_{i},\sS_{X}\thickslash\Gamma_{X}}(\xi_{i})|_{T_{ij}}\ar[r]&1}\]
which shows that $\rho\circ\psi\circ g_{i}=\emph{id}\in\sA
ut_{T_{i},\sS_{X}\thickslash\Gamma_{X}}(\xi_{i})|_{T_{ij}}$, hence
we can glue $\{\xi_{i}\}$ and get an arrow $\sigma:
T\lra\sS_{X}\thickslash\Gamma_{X}$. The difference of
two choices of $\{f_{ij}\}$ is an element in $\sA
ut_{T_{i},\sA}(\zeta_{i})|_{T_{ij}}$, and becomes the
identity when composed with $\rho\circ\psi$. Hence the arrow
$\sigma: T\lra\sS_{X}\thickslash\Gamma_{X}$ is unique.

The last step is to show that the following diagram is cartesian:
$$\xymatrix{
      \sA\ar[rr]^-{\psi}\ar[d]_{\pi}&&\sS_{X}\ar[d]^{\rho}\\
      T \ar[rr]^-{\sigma}          &&\sS_{X}\thickslash\Gamma_{X}.
}$$
This is equivalent to showing that the following is
cartesian:
$$\xymatrix{
      \sA_{i}\ar[rr]^-{\psi}\ar[d]_{\pi_{i}}&&\sS_{X}\ar[d]^{\rho}\\
      T_{i} \ar[rr]^-{\sigma}
      &&\sS_{X}\thickslash\Gamma_{X},
}$$ where $\sA_{i}\lra T_{i}$ is obtained by pulling back $\sA\lra
T$ along the covering $T_{i}\lra T$. We also note that
$\pi_{i}$ has a section $\zeta_{i}': T_{i}\lra\sA_{i}$, which is
obtained from the following diagram:
$$\xymatrix{T_{i}\ar@/^/[drrr]^{\zeta_{i}}\ar@{-->}[dr]^--{\zeta_{i}'}\ar@/_/[ddr]_{\emph{id}}&&& \\
            & \sA_{i}\ar[rr]\ar[d]^{\pi_{i}}&& \sA\ar[d] \\
            & T_{i}\ar[rr]        && T .\\
}$$

Now, assume we are given a $T_{i}$-scheme $U$ and the diagram
$$\xymatrix{U\ar@/^/[drrr]^{\xi}\ar@{-->}[dr]^--{h}\ar@/_/[ddr]_{f}&&& \\
            & \sA_{i}\ar[rr]^{\psi}\ar[d]_{\pi_{i}}&& \sS_{X}\ar[d]^{\rho} \\
            & T_{i}\ar[rr]^-{\sigma}\ar@/_1pc/[u]_{\zeta_{i}'}        && \sS_{X}\thickslash\Gamma_{X}, \\
}$$ where $\xi=(\xi_{U},G_{U})\in\sS_{X}(U)$.

We need to show that $h=\zeta_{i}'\circ f$ is the unique arrow
making the diagram commutative. Note that $\pi_{i}\circ
h=\pi_{i}\circ\zeta_{i}\circ f=\emph{id}_{T_{i}}\circ f=f$. So we
only need to show that $\xi=\psi\circ h$. But by our construction,
$\sigma\circ f=\rho\circ\psi\circ\zeta_{i}'\circ
f=\rho\circ\psi\circ h$, so $\rho\circ\psi\circ h=\rho\circ\xi$.
Since our construction is by gluing \'etale locally, we can
restrict to a small \'etale neighborhood, and assume that $\rho$
is an isomorphism, hence $\psi\circ h=\xi$. Since both $\pi_{i}$
and $\zeta_{i}'$ are \'etale and surjective, $h$ is unique by our
construction.
\end{proof}

\begin{rem}
\begin{enumerate}
\item In the proof of this proposition, we actually proved that
every pair $(\sA,\psi)$ over a scheme $T$, where $\sA$ is an \'etale
gerbe over T and $\psi: \sA\lra\sS_{X}$ is a representable arrow,
corresponds to an arrow $T \lra \sS_{X}\thickslash\Gamma_{X}$. Also,
$\sS_X$ is the universal gerbe over
$\sS_{X}\thickslash\Gamma_{X}$.

\item We can identify $\sG_X$ with
$\sS_{X}\thickslash\Gamma_{X}$. By Theorem \ref{rigid}, $\sG_{X}$ is a
DM-stack.
\end{enumerate}
\end{rem}

%
%
\section{Very twisted curves and their maps}
\label{sec:3}


Recall from \cite{AGV} that a twisted stable map over a scheme $T$ is given by a family of twisted curves $\sC\lra T$ and $n$ gerbes $\Sigma_i\subset \sC$, with a representable morphism $f:\sC\lra X$.  $\KSTACK{g}{n}{X}{\beta}$ is the moduli stack parameterizing maps of this type with the image of $f$ lying in the curve class $\beta\in H^\ast(X)$.  Our goal is to extend the notion of stable maps into $X$ by allowing the source twisted curve $\sC$ to have generic stabilizers.  This leads us to the following definitions:

\begin{defn}
A \emph{very twisted curve} over a scheme $T$ is an \'etale gerbe $\fC\lra\sC$ over a twisted curve $\Sigma_i\subset\sC\lra T$.  We define a \emph{very twisted stable map} to be a diagram of the form:
\[
\xymatrix{
**[l]\widetilde{\Sigma_i}\subset\fC \ar[r]^{\tilde{f} (rep.)} \ar[d]& X\\
 **[l]\Sigma_i\subset\sC \ar[d]\\
 T\\}
\]
where $\fC\lra\sC\lra T$ is a very twisted curve, the markings $\widetilde{\Sigma_i}\subset\fC$ are given by taking the preimage of $\Sigma_i$, and such that the diagram admits finitely many automorphisms (our stability condition). Stability is equivalent to the corresponding map $\sC\lra\sG_X$ being a twisted stable map, and $\tilde{f}$ factoring through the projection $\fC\lra\sS_X$.
\end{defn}

\begin{rem}
Notice that in our definitions we are using a single gerbe structure over the whole curve.  We are not considering curves with different generic stabilizers over different components.
\end{rem}  

Given a curve class $\beta\in H^\ast(X)$, we want to define the moduli stack $\vkstack$ parametrizing very twisted stable maps with image class $\beta$.  A map $\sC \to \sG_X$ from a twisted curve $\sC$ to our stack $\sG_X$ corresponds to a diagram of the form:
        \[
    \xymatrix{
        \fC \ar[r]^{(rep.)} \ar[d]&X\\
    \sC.\\}
        \]
    Where $\fC \to \sC$ is an \'etale gerbe.  As stated above, \'etale gerbes over twisted curves are our very twisted curves.  Let $\{\beta_j\}\subset H^\ast(\sG_\sX)$ be the collection of curve classes satisfying $\phi_\ast\alpha^\ast\beta_j=\beta$.  We are led to consider, for each j, the diagram
\[
\xymatrix{
**[l] \widetilde{\Sigma_i}\subseteq\fC_j \ar[d]  \ar[r] \ar@/^1pc/[rr]^{\tilde f} & \sS_X \ar[d]^\alpha \ar[r]_\phi & X\\
**[l] \Sigma_i\subseteq\sC_j \ar[r]^f \ar[d]^\pi& \sG_X &\\
\KSTACK{g}{n}{\sG_X}{\beta_j}.\\}
\]
Here $\sC_j$ is the universal twisted curve sitting above $\KSTACK{g}{n}{\sG_X}{\beta_j}$ and $\fC_j$ is the universal very twisted curve given by pulling back.  The compositions:
\[
\xymatrix{
**[l] \widetilde{\Sigma_i}\subseteq\fC_j \ar[d]  \ar[r]  & X\\
\KSTACK{g}{n}{\sG_X}{\beta_j}\\}
\]
give us a parametrization of very twisted stable maps
\[
\xymatrix{
\widetilde C \ar[r]^{(rep)} \ar[d] & X\\
C \ar[d]\\
T\\}
\]
whose images lie in $\beta$ and whose corresponding map $C\lra\sG_X$ has image lying in $\beta_j$.  We get our desired moduli space by taking a disjoint union over all $j$:
\begin{defn} With $X$ as above, let
\[
    \vkstack := \coprod_{j} \KSTACK{g}{n}{\sG_X}{\beta_j}.
\]
\end{defn}
This moduli space has all the nice properties we want given for free by the construction of $\kstack$.  Sitting above it are two universal objects, one corresponding to very twisted stable maps into $X$ by taking $\widetilde\sC$ over each $\KSTACK{g}{n}{\sG_X}{\beta_\sG}$, and the other corresponding to twisted stable maps into $\sG_X$ by taking $\sC$.  We use the following notation to distinguish them:

\begin{align*}
\xymatrix{
\coprod_j \sC_j \ar[d]^\pi \ar[r]^f & \sG_X\\
\coprod_{j}\KSTACK{g}{n}{\sG_X}{\beta_j}\\}
& &\textrm{and}& &
\xymatrix{
\coprod_j \fC_j \ar[d]^{\tilde\pi} \ar[r]^{\tilde f} & X\\
\vkstack.\\}
\end{align*}

%
%

\section{Equality of the Gromov-Witten Theories}
\label{sec:4}


%
%

As discussed in section \ref{sec:3}, we have moduli spaces with two different universal curves, each part of the diagram
\[
\xymatrix{
\coprod_j \fC_j \ar[d]^\delta \ar@/_2pc/[dd]_{\tilde\pi} \ar[r] \ar@/^2pc/[rr]^{\tilde f} & \sS_X \ar[d] \ar[r] & X\\
\coprod_j \sC_j \ar[r]^f \ar[d]^\pi& \sG_X &\\
\vkstack \ar[d]\\
\mathfrak{M}^{tw}_{g,n}.}
\]
In this section, we show that the construction of the virtual fundamental classes and the evaluation maps give the same answer for $\coprod_{j}\KSTACK{g}{n}{\sG_X}{\beta_j}$ and $\vkstack$.  The virtual fundamental classes are constructed by looking at the relative obstruction theories $(R\pi_\ast f^\ast T_{\sG_X})^\vee$ and $(R\tilde\pi_\ast\tilde f^\ast T_{X})^\vee$ over $\mathfrak{M}^{tw}_{g,n}$ as in \cite{B}, \cite{BF} and \cite{AGV}.  We simply need to show that $R\pi_\ast f^\ast T_{\sG_X} \cong R\tilde\pi_\ast\tilde f^\ast T_{X}$.  Since push-forward is functorial, it will be enough to show that $f^\ast T_{\sG_X} \cong \delta_\ast\tilde f^\ast T_{X}$, for each $j$, in the diagram:
\[
\xymatrix{
\fC_j \ar[d]^\delta \ar[r] \ar@/^2pc/[rr]^{\widetilde f} & \sS_X \ar[d] \ar[r] & X\\
\sC_j \ar[r]^f & \sG_X .\\}
\]

  To do this, we prove an analogue to the tangent bundle lemma (Lemma 3.6.1) of \cite{AGV}.  We follow their proof closely.

\begin{lem}
Let $S$ be a scheme, and $f:S\lra \sG_X$ a morphism with its associated diagram
\[
\xymatrix{
\sA \ar[d]^\delta \ar@/^/[r]^{\widetilde f} & X\\
S .\\}
\]
Then there is a canonical isomorphism $\delta_\ast(\widetilde f^\ast(T_X)) \cong f^\ast(T_{\sG_X})$.
\end{lem}
\begin{proof}
We showed above that by rigidification we have the universal gerbe $\sS_X \lra \sG_X$.  Consider the following diagram of smooth stacks:
\[
\xymatrix{
\sS_X \ar[d]^\omega \ar[r]^F & X\\
\sG_X .}
\]
Given $f:S\lra \sG_X$ as in the statement of the lemma, we have the diagram
\[
\xymatrix{
\sA \ar[r]^g \ar[d]^\pi \ar@/^2pc/[rr]^{\widetilde f}& \sS_X \ar[d]^\omega \ar[r]^F & X\\
S \ar[r]^f & \sG_X .}
\]
The morphism $\omega$ is flat and, by Lemma 2.3.4 in \cite{AV}, $\omega_\ast$ is exact on coherent sheaves, so for any locally free sheaf $\sF$ on $\sS_X$ we have $f^\ast\omega_\ast\sF = \pi_\ast g^\ast\sF$.  All that is needed is to check that $\omega_\ast F^\ast T_X \cong T_{\sG_X}$.

By applying $\omega_\ast$ to the natural morphism $T_{\sS_X}\lra F^\ast T_X$, we get $\omega_\ast T_{\sS_X}\lra \omega_\ast F^\ast T_X$.  But $\omega: \sS_X\lra \sG_X$ is an \'etale gerbe, so $\omega_\ast T_{\sS_X}\cong T_{\sG_X}$.  Through this isomorphism, we get a morphism $T_{\sG_X}\lra \omega_\ast F^\ast T_X$.  The claim is that this is an isomorphism.  We show this by looking at the geometric points.

The fiber of a geometric point $y$ of $\sG_X$ through $\omega$ can be identified with $BH$ for some subgroup $H \subset G_x$.  Here $x$ is a geometric point of $X$ with stabilizer group $G_x$.  This lifts $y$ to $\sS_X$, mapping $y$ to $x$.  Let $T$ be the pullback of $T_X$ to our lift of $y$.  There is a natural action of $H$ on $T$, and the fiber of $\omega_\ast F^\ast T_X$ at $y$ is given by the space of invariants $T^H$.  We must show that we get an isomorphism of this fiber with the fiber of $T_{G_X}$ at $y$.

We can view $X$ in a local chart around $x$ as $[U/G_x]$.  Then $\sS_X$ has a local chart $[U^H/N(H)]$, where we quotient out by the normalizer subgroup $N(H)$.  But $T_{U^H}=T_U^H$ and $\sS_X\lra \sG_X$ is an \'etale gerbe, so we obtain our desired isomorphism $T_{G_X,y}\cong T^H$.
\end{proof}

\begin{cor}
The virtual fundamental classes $\left[\coprod_j\KSTACK{g}{n}{\sG_X}{\beta_j}\right]^{vir}$ and $\left[\vkstack\right]^{vir}$ are equal.  In particular, for $n=0$ we have equality of the Gromov-Witten invariants:
\[
\int\limits_{\left[\coprod_j\KSTACK{g}{n}{\sG_X}{\beta_j}\right]^{vir}}\!\!\!\!\!\!\!\!\!\!\!\!\!\!\!\!\!\!1\,\,\,\,\,\,\,\,\,\,\,\, = \int\limits_{\virclass{\vkstack}}\!\!\!\!\!\!\!\!\!\!\!\!1.
\]
\end{cor}
\begin{proof}
Immediate from the lemma.
\end{proof}

In general, our Gromov-Witten invariants will be defined by taking $\gamma_i\in A^\ast (\overline{\sI}_\mu(\sG_X))_\Q$, pulling back along the evaluation maps $e_i:\coprod_j\KSTACK{g}{n}{\sG_X}{\beta_j}\lra \overline{\sI}_\mu(\sG_X)$, and taking a product:
    \[
    \left<\gamma_1,\ldots,\gamma_n\right>^X_{g,\beta_j} = \left(\prod_{i=1}^n e_i^\ast\gamma_i \right)\cap \left[\coprod_j\sK_{g,n}(\sG_X,\beta_j)\right]^{vir}.
    \]
    Recall from \cite{AGV} that $\overline{\sI}_\mu(\sG_X)$ is given by the disjoint union
\[
\overline{\sI}_\mu(\sG_X):= \coprod_r \overline{\sI}_{\mu_r}(\sG_X).
\]
Each component $\overline{\sI}_{\mu_r}(\sG_X)$ can be reinterpreted as parametrizing diagrams of the form
\[
\xymatrix{
\sA \ar[d] \ar[r]^\phi & X\\
\Sigma \ar[d] \\
T\\}
\]
where $\Sigma\lra T$ is a gerbe over a scheme $T$ banded by $\mu_r$ and $\phi$ is a representable morphism.  This allows us to extent our evaluation map to a map $\widetilde{e}_i:\vkstack\lra\overline{\sI}_\mu(\sG_X)$.  Given a very twisted stable map $\widetilde f:\fC\lra X$ over $T$, its image $\widetilde e_i(\widetilde f)\in \overline{\sI}_\mu(\sG_X)(T)$ is the object associated with the diagram:
\[
\xymatrix{
\widetilde\Sigma_i \ar[d] \ar[r]^{\widetilde f|_{\widetilde\Sigma_i}} & X\\
\Sigma_i \ar[d] \\
T.\\}
\]
One can see that this agrees with the image $e_i(f)\in \overline{\sI}_\mu(\sG_X)(T)$ of the twisted stable map
$f:\sC\lra\overline{\sI}_\mu(\sG_X)$ given by the diagram:
\[
\xymatrix{
\Sigma_i \ar[d] \ar[r]^{f|_{\Sigma_i}} & \sG_X\\
T.\\}
\]
Our desired equality of Gromov-Witten invariants follows.

\begin{rem}
Although the naive GW-theory described above is the natural first approach to take, there is evidence from the
DT-side of the conjectural correspondence for orbifolds that some adjustment to the theory is needed.
In particular, the moduli spaces for orbifold Calabi-Yau 3-folds in the DT-theory have virtual dimension 0, and
the correspondence would hope for the same to be true on the GW-side.  An example sugested by Jim Bryan
is the following.  Let $E$ be the total space of $\sO_{\PP^1}(-1)\bigoplus \sO_{\PP^1}(-1)$ over $\PP^1$. The quotient $X=\left[ E/(\Z/2\Z)\right]$, where the action along each fiber is component-wise, is an
orbi-Calabi-Yau 3-fold with the zero-section giving an embedded $\widetilde{\PP^1}=\left[\PP^1/(\Z/2\Z)\right]$.
In this case, $\sS_X = X \sqcup \widetilde{\PP^1}$ and $\sG_X = X \sqcup \PP^1$.  Stable maps into $\sG_X$
(in particular, into $\PP^1$) will not have virtual dimension zero.  So even in this straightforward case we have
a virtual fundamental class $\left[\vkstack\right]^{vir}$ with non-zero virtual dimension.  A nicer theory might
adjust the relative obstruction theory used to construct the virtual fundamental class in order to fix the problem
with the virtual dimension.  This is something we hope to look into.
\end{rem}
\bibliographystyle{amsalpha}             
\bibliography{myrefs}       
\end{document}